\documentclass[a4paper,11pt]{article}

\usepackage{amsmath}
\usepackage{amsfonts}
\usepackage{multind}
\usepackage{supertabular}
\usepackage[latin9]{inputenc}
\usepackage[english]{varioref}
\usepackage{dcolumn}
\usepackage[height=22cm , width = 16cm , top = 4cm , left = 3cm, a4paper]{geometry}
\usepackage[a4paper]{geometry}
\usepackage[final]{graphicx}
\usepackage{epsfig}
\usepackage{pstricks}
\usepackage{psfrag}
\usepackage{rotating}
\usepackage{supertabular}
\usepackage{booktabs}
\usepackage{delarray}
\usepackage{rotating}
\usepackage{subfigure}
\usepackage{nextpage}
\usepackage{layout}
\usepackage{amsthm}
\usepackage{dsfont}
\usepackage{hyperref}
\usepackage{color}
{                     
{                     
{                       
{                       

\theoremstyle{plain}
\newtheorem{thm}{Theorem}[section]

\newtheorem{lem}[thm]{Lemma}

\newtheorem{defn}[thm]{Definition}
\newtheorem{rmkk}[thm]{Remark}
\newtheorem{hypp}[thm]{Hypotheses}
\parindent0cm
\newcommand{\enter}{\bigskip}

\begin{document}
\thispagestyle{empty}

 \author{{Ankik Kumar Giri $^{1,3}$\footnote{Corresponding author. Tel +43 (0)3842-402-1706; Fax +43 (0)3842-402-1702 \newline{\it{${}$ \hspace{.3cm} Email address: }}ankik-kumar.giri@unileoben.ac.at, ankik.giri@ovgu.de},
 Philippe Lauren\c{c}ot$^{2}$, Gerald Warnecke$^{3}$}\vspace{.2cm}\\
\footnotesize \it $^1$ Institute for Applied Mathematics, Montan University Leoben,\\
\footnotesize \it Franz Josef Stra{\ss}e 18, A-8700 Leoben, Austria \vspace{.2cm} \\
\footnotesize \it $^2$ Institut de Math\'ematiques de Toulouse, CNRS UMR~5219, Universit\'e de Toulouse, \vspace{-.2cm}\\
 \footnotesize \it F--31062 Toulouse Cedex 9, France, \vspace{.2cm} \\
\footnotesize \it $^3$Institute for Analysis and Numerics, Otto-von-Guericke University Magdeburg, \vspace{-.2cm}\\
 \footnotesize \it Universit\"{a}tsplatz 2, D-39106 Magdeburg, Germany \vspace{.2cm} \\
  }

\title{Weak Solutions to the Continuous Coagulation Equation with Multiple Fragmentation}
\maketitle

\hrule \vskip 8pt

\begin{quote}
{\small {\em\bf Abstract} The existence of weak solutions to the continuous coagulation equation
with multiple fragmentation is shown for a class of unbounded coagulation and fragmentation kernels, the fragmentation kernel having possibly a singularity at the origin. This result extends previous ones where either boundedness of the coagulation kernel or no singularity at the origin for the fragmentation kernel were assumed.\enter
}
\end{quote}

{\bf Keywords:} Coagulation; Multiple Fragmentation;  Unbounded kernels; Existence; Weak compactness

\vskip 10pt \hrule

\section{Introduction}\label{existintroduction1}

The continuous coagulation and multiple fragmentation equation describes the evolution of the number density $f=f(x,t)$ of particles of volume $x\geq 0$ at time $t\geq 0$ and reads
\begin{align}\label{cfe1}
\frac{\partial f(x,t)}{\partial t}  = & \frac{1}{2}\int_{0}^{x} K(x-y,y)f(x-y,t)f(y,t)dy - \int_{0}^{\infty} K(x,y)f(x,t)f(y,t)dy\nonumber\\  &+ \int_{x}^{\infty} b(x,y)S(y)f(y,t)dy-S(x)f(x,t),
\end{align}
with
\begin{align}\label{in1}
f(x,0) = f_{0}(x)\geq 0 .
\end{align}
The first two terms on the right-hand side of (\ref{cfe1}) accounts for the formation and disappearance of particles as a result of coagulation events and the coagulation kernel $K(x,y)$ represents the rate at which particles of volume $x$ coalesce with particles of volume $y$. The remaining two terms on the right-hand side of (\ref{cfe1}) describes the variation of the number density resulting from fragmentation events which might produce more than two daughter particles, and the breakage function $b(x,y)$ is the probability density function for the formation of particles of volume $x$ from the particles of volume $y$. Note that it is non-zero only for $x<y$. The selection function $S(x)$ describes the rate at which particles of volume $x$ are selected to fragment. The selection function $S$ and breakage function $b$ are defined in terms of the multiple-fragmentation kernel $\Gamma$ by the identities
\begin{align}\label{Selection rate defi1}
S(x)=\int_{0}^{x}\frac{y}{x}\ \Gamma(x,y)dy,\ \ \ \ b(x,y)=\Gamma(y,x)/ S(y).
\end{align}
The breakage function is assumed here to have the following properties
\begin{align}\label{N1}
\int_{0}^{y}b(x,y)dx = N< \infty, \ \ \text{for all}\ \ y>0,\ \ \ \ \ b(x,y)=0\ \ \text{for}\ \ x> y,
\end{align}
and
\begin{align}\label{mass1}
\int_{0}^{y}xb(x,y)dx = y\ \ \text{for all}\ \ y>0.
\end{align}
The parameter $N$ represents the number of fragments obtained from the breakage of particles of volume $y$ and is assumed herein to be finite and independent of $y$. This is however inessential for the forthcoming analysis, see Remark~\ref{re:N} below. As for the condition (\ref{mass1}), it states that the total volume of the fragments resulting from the splitting of a particle of volume $y$ equals $y$ and thus guarantees that the total volume of the system remains conserved during fragmentation events.\\
The existence of solutions to coagulation-fragmentation equations has already been the subject of several papers which however are mostly devoted to the case of binary fragmentation, that is, when the fragmentation kernel $\Gamma$ satisfies the additional symmetry property $\Gamma(x+y,y)=\Gamma(x+y,x)$ for all $(x,y)\in ]0,\infty[^2$, see the survey \cite{LM:2004} and the references therein. The coagulation-fragmentation equation with multiple fragmentation has received much less attention over the years though it is already considered in the pioneering work \cite{Melzak:1957}, where the existence and uniqueness of solutions to (\ref{cfe1})-(\ref{in1}) are established for bounded coagulation and fragmentation kernels $K$ and $\Gamma$. A similar result was obtained later on in \cite{McLaughlin:1997II} by a different approach. The boundedness of $\Gamma$ was subsequently relaxed in \cite{Lamb:2004} where it is only assumed that $S$ grows at most linearly, but still for a bounded coagulation kernel. Handling simultaneously unbounded coagulation and fragmentation kernels turns out to be more delicate and, to our knowledge, is only considered in \cite{Laurencot:2000} for coagulation kernels $K$ of the form $K(x,y)=r(x) r(y)$ with no growth restriction on $r$ and a moderate growth assumption on $\Gamma$ (depending on $r$) and in \cite{GIRI:2010EXT} for coagulation kernels satisfying $K(x,y)\le \phi(x)\phi(y)$ for some sublinear function $\phi$ and a moderate growth assumption on $\Gamma$ (see also \cite{PhL:2002} for the existence of solutions for the corresponding discrete model). Still, the fragmentation kernel $\Gamma$ is required to be bounded near the origin in \cite{GIRI:2010EXT,Laurencot:2000} which thus excludes kernels frequently encountered in the literature such as $\Gamma(y,x)=(\alpha+2)\ x^\alpha\ y^{\gamma-(\alpha+1)}$ with $\alpha>-2$ and $\gamma\in\mathbb{R}$ \cite{McGradyZiff:1987}.

The purpose of this note is to fill (at least partially) this gap and establish the existence of weak solutions to (\ref{cfe1}) for simultaneously unbounded coagulation and fragmentation kernels $K$ and $\Gamma$, the latter being possibly unbounded for small and large volumes. More precisely, we make the following hypotheses on the coagulation kernel $K$, multiple-fragmentation kernel $\Gamma$, and selection rate $S$.
\begin{hypp}\label{hyp1}
(H1) $K$ is a non-negative measurable function on $[0,\infty[ \times [0,\infty[$ and is symmetric, i.e. $K(x,y)=K(y,x)$ for all $x,y \in ]0,\infty[$,\\
\\
(H2) $K(x,y)\leq \phi(x)\phi(y)$ for all $x,y\in ]0,\infty[$ where $\phi(x)\leq k_1(1+x)^{\mu}$ for some $0\leq \mu< 1$ and constant $k_1>0$.\\
\\
(H3) $\Gamma$  is a non-negative measurable function on $]0,\infty[ \times ]0,\infty[$ such that $\Gamma(x,y)=0$ if $0<x<y$. Defining $S$ and $b$ by (\ref{Selection rate defi1}), we assume that $b$ satisfies (\ref{mass1}) and there are $\theta\in [0,1[$ and two non-negative functions $k:]0,\infty[\to [0,\infty[$ and $\omega:]0,\infty[^2\to [0,\infty[$ such that, for each $R\geq 1$:\\
\\
(H4) we have $\Gamma(y,x) \leq k(R)\ y^{\theta}$ for  $y>R$ and $x\in ]0,R[$,\\
\\
(H5) for $y \in ]0,R[$ and any measurable subset $E$ of $]0,R[$, we have
\begin{align*}
\int_{0}^{y} \mathds{1}_E(x) \Gamma(y,x)dx \leq \omega(R,|E|), \ \ y \in ]0,R[,
\end{align*}
where $|E|$ denotes the Lebesgue measure of $E$, $\mathds{1}_E$ is the indicator function  of $E$ given by
\begin{equation*}
\mathds{1}_E(x):=\begin{cases}
1\ \, & \text{if}\ x\in E, \\
0\ \, &  \text{if}\ x\notin E,
\end{cases}
\end{equation*}
and we assume in addition that
\begin{align*}
\lim_{\delta \to 0} \omega(R, \delta)=0,
\end{align*}
\\
(H6) $S \in L^{\infty}]0,R[$.
\end{hypp}

We next introduce the functional setting which will be used in this paper: define the Banach space $X$ with norm $\|\cdot\|$ by
\begin{align*}
X=\{f\in L^1(0,\infty):\|f\|< \infty\}\ \ \mbox{where}\ \ \|f\|=\int_{0}^{\infty}(1+x)|f(x)|dx,
\end{align*}
together with its positive cone
\begin{align*}
 X^+=\{f\in X: f\geq 0 \ \ a.e.\}.
 \end{align*}
For further use, we also define the norms
\begin{align*}
\|f\|_x=\int_{0}^{\infty}x|f(x)|dx \ \ \text{and}\ \ \|f\|_1=\int_{0}^{\infty}|f(x)|dx, \qquad f\in X.
\end{align*}

\medskip

The main result of this note is the following existence result:

\begin{thm}\label{existmain theorem1}
Suppose that (H1)--(H6) hold and assume that $f_0\in X^+$. Then (\ref{cfe1})-(\ref{in1}) has a weak solution $f$ on $]0,\infty[$ in the sense of Definition~\ref{def1} below. Furthermore, $\|f(t)\|_x\le \|f_0\|_x$ for all $t\ge 0$.
\end{thm}

Before giving some examples of coagulation and fragmentation kernels satisfying (H1)--(H6), we recall the definition of a weak solution to (\ref{cfe1})-(\ref{in1}) \cite{Stewart:1990I}.

\begin{defn}\label{def1} Let $T \in ]0,\infty]$. A solution $f$ of (\ref{cfe1})-(\ref{in1}) is a non-negative function $f: [0,T[\to X^+$ such that, for a.e. $x\in ]0,\infty[$ and all $t\in [0,T[$,\\
\\
      (i) $s\mapsto f(x,s)$ is continuous on $[0,T[$,\\
\\
     (ii)  the following integrals are finite\\
     \begin{align*}
     \int_{0}^{t}\int_{0}^{\infty}K(x,y)f(y,s)dyds<\infty\ \ and \ \ \int_{0}^{t}\int_{x}^{\infty}b(x,y)S(y)f(y,s)dyds<\infty,
     \end{align*}
\\
      (iii)  the function $f$ satisfies the following weak formulation of (\ref{cfe1}){-(\ref{in1})}\\
     \begin{align*}
     f(x,t)&=f_0(x)+\int_{0}^{t}\left\{ \frac{1}{2}\int_{0}^{x}K(x-y,y)f(x-y,s)f(y,s)dy \right.\nonumber\\
     &\left. -\int_{0}^{\infty}K(x,y)f(x,s)f(y,s)dy +\int_{x}^{\infty}b(x,y)S(y)f(y,s)dy - S(x)f(x,s)\right\}ds.
     \end{align*}
\end{defn}

Coming back to (H1)-(H6), it is clear that coagulation kernels satisfying $K(x,y)\le x^\mu y^\nu + x^\nu y^\mu$ for some $\mu\in [0,1[$ and $\nu\in [0,1[$ which are usually used in the mathematical literature satisfy (H1)-(H2), see also \cite{GIRI:2010EXT} for more complex choices. Let us now turn to fragmentation kernels which also fit in the classes considered in Hypotheses \ref{hyp1}.

Clearly, if we assume that
$$
\Gamma\in L^{\infty}(]0,\infty[\times ]0,\infty[)
$$
 as in \cite{GIRI:2010EXT,McLaughlin:1997II}, (H4) and (H5) are satisfied with $k=\|\Gamma\|_{L^\infty}$, $\theta=0$, and $\omega(R,\delta)=\|\Gamma\|_{L^\infty} \delta$. Now let us take
\begin{equation}
S(y) =y^\gamma \ \ \text{and}\ \ b(x,y) = \frac{\alpha+2}{y} \left( \frac{x}{y} \right)^\alpha\ \ \text{for}\ \ 0<x<y, \label{volvic}
\end{equation}
where $\gamma > 0$ and $\alpha \geq 0$, see \cite{McGradyZiff:1987,Peterson:1986}. Then
\begin{align*}
\Gamma(y,x)= (\alpha+2) x^{\alpha} y^{\gamma-(\alpha+1)} \ \ \text{for}\ \ 0<x<y.
\end{align*}
Let us first check (H5). Given $R>0$, $y\in ]0,R[$, and a measurable subset $E$ of $]0,R[$, we deduce from H\"{o}lder's inequality that
\begin{align*}
\int_{0}^{y} \mathds{1}_E(x) \Gamma(y,x)dx & = (\alpha+2) y^{\gamma-(\alpha+1)}\int_{0}^{y} \mathds{1}_E(x)x^{\alpha} dx\\
& \leq (\alpha+2) y^{\gamma-(\alpha+1)} |E|^{\frac{\gamma}{\gamma+1}} \bigg( \int_{0}^{y} x^{\alpha (\gamma+1)} dx \bigg)^{\frac{1}{\gamma+1}}\\
& \leq (\alpha+2) |E|^{\frac{\gamma}{\gamma+1}} (1+\alpha(\gamma+1))^{-\frac{1}{\gamma+1}} y^{\alpha+ \frac{1}{\gamma+1}+ \gamma-(\alpha+1)}\\
& \leq C(\alpha, \gamma) y^{\frac{{\gamma}^2}{\gamma+1}} |E|^{\frac{\gamma}{\gamma+1}}\\
& \leq C(\alpha, \gamma) R^{\frac{{\gamma}^2}{\gamma+1}} |E|^{\frac{\gamma}{\gamma+1}}.
\end{align*}
This shows that (H5) is fulfilled with $\omega(R,\delta)=C(\alpha, \gamma) R^{\frac{{\gamma}^2}{\gamma+1}} \delta^{\frac{\gamma}{\gamma+1}}$. As for (H4), for $0<x<R<y$, we write
\begin{align*}
 \Gamma(y,x) \leq (\alpha+2)\ R^\alpha\ y^{\gamma-(\alpha+1)} \leq \begin{cases}
(\alpha+2)\ R^{\gamma-1}\ \, & \text{if}\ \gamma\le \alpha+1, \\
(\alpha+2)\ R^\alpha\ y^{\gamma-(\alpha+1)}\ \, &  \text{if}\ \gamma>\alpha+1,
\end{cases}
\end{align*}
and (H4) is satisfied provided $\gamma < 2 + \alpha$ with $k(R)=(\alpha+2)\ R^{\gamma-1}$ and $\theta=0$ if $\gamma\in ]0,\alpha+1]$ and $k(R)=(\alpha+2)\ R^\alpha$ and $\theta=\gamma-(\alpha+1) \in [0,1[$ if $\gamma\in ]\alpha+1,\alpha+2[$. Therefore, Theorem~\ref{existmain theorem1} provides the existence of weak solutions to (\ref{cfe1})-(\ref{in1}) for unbounded coagulation kernels $K$ satisfying (H1)-(H2) and multiple fragmentation kernels $\Gamma$ given by (\ref{volvic}) with $\alpha\ge 0$ and $\gamma\in ]0,\alpha+2[$. Let us however mention that some fragmentation kernels which are bounded at the origin and considered in \cite{GIRI:2010EXT,Laurencot:2000} need not satisfy (H4)-(H5).

\begin{rmkk}\label{remark:r1}
While the requirement $\gamma<\alpha+2$ restricting the growth of $\Gamma$ might be only of a technical nature, the constraint $\gamma>0$ might be more difficult to remove. Indeed, it is well-known that there is an instantaneous loss of matter in the fragmentation equation when $S(x)=x^{\gamma}$ and $\gamma<0$ produced by the rapid formation of a large amount of particles with volume zero (dust), a phenomenon refered to as disintegration or shattering \cite{McGradyZiff:1987}. The case $\gamma=0$ thus appears as a borderline case.
\end{rmkk}

Let us finally outline the proof of Theorem~\ref{existmain theorem1}. Since the pioneering work \cite{Stewart:1990I}, it has been realized that $L^1$-weak compactness techniques are a suitable way to tackle the problem of existence for coagulation-fragmentation equations with unbounded kernels. This is thus the approach we use hereafter, the main novelty being the proof of the estimates needed to guarantee the expected weak compactness in $L^1$. These estimates are derived in Section~\ref{subs:wk} on a sequence of unique global solutions to truncated versions of (\ref{cfe1})-(\ref{in1}) constructed in Section~\ref{subsec trunc1}. After establishing weak equicontinuity with respect to time in Section~\ref{subs:equit}, we extract a weakly convergent subsequence in $L^1$ and finally show that the limit function obtained from the weakly convergent subsequence is actually a solution to (\ref{cfe1})-(\ref{in1}) in Sections~\ref{subs:limit} and~\ref{subs:exist}.  

\section{Existence}\label{existexistence1}

\subsection{Approximating equations}\label{subsec trunc1}

In order to prove the existence of solutions to (\ref{cfe1}-\ref{in1}), we take the limit of a sequence of approximating equations obtained by replacing the kernel $K$ and selection rate $S$ by their ``cut-off'' analogues $K_n$ and $S_n$ \cite{Stewart:1990I}, where
\begin{equation*}
K_n(x,y):=\begin{cases}
K(x,y)\ \, & \text{if}\ x+y< n, \\
\text{0}\ \, &  \text{if}\ x+y\geq n,
\end{cases}
\qquad
S_n(x):=\begin{cases}
S(x)\ \, & \text{if}\ 0<x< n, \\
\text{0}\ \, &  \text{if}\ x\geq n,
\end{cases}
\end{equation*}
for $n\ge 1$. Owing to the boundedness of $K_n$ and $S_n$ for each $n\ge 1$, we may argue as in \cite[Theorem 3.1]{Stewart:1990I}  or \cite{Walker:2002} to show that the approximating equation
\begin{align}\label{trunc1}
\frac{\partial f^n(x,t)}{\partial t}  = & \frac{1}{2}\int_{0}^{x} K_n(x-y,y)f^n(x-y,t)f^n(y,t)dy - \int_{0}^{n-x} K_n(x,y)f^n(x,t)f^n(y,t)dy\nonumber\\  &+\int_{x}^{n}b(x,y) S_n(y)f^n(y,t)dy - S_n(x)f^n(x,t),
\end{align}
with initial condition
\begin{equation}\label{trunc in1}
f^{n}_0(x):=\begin{cases}
f_0(x)\ \, & \text{if}\ 0 < x< n, \\
\text{0}\ \, &  \text{if}\ x\geq n.
\end{cases}
\end{equation}
has a unique non-negative solution $f^n\in C^1([0,\infty[;L^1]0,n[)$ such that $f^n(t)\in X^+$ for all $t\ge 0$. In addition, the total volume remains conserved for all $t\in[0,\infty[$, i.e.\ \\
\begin{align}\label{trunc mass1}
\int_{0}^{n}xf^n(x,t)dx=\int_{0}^{n}xf^n_0(x)dx.
\end{align}

From now on, we extend $f^n$ by zero to $]0,\infty[\times [0,\infty[$, i.e. we set $f^n(x,t)=0$ for $x>n$ and $t\ge 0$. Observe that we then have the identity $S_n f^n = S f^n$.

\medskip

Next, we need to establish suitable estimates in order to apply the \textit{Dunford-Pettis Theorem} \cite[Theorem 4.21.2]{Edwards:1965} and then the equicontinuity of the sequence $(f^n)_{n\in \mathds{N}}$ in time to use the \textit{Arzel\`{a}-Ascoli Theorem} \cite[Appendix A8.5]{Ash:1972}. This is the aim of the next two sections.

\subsection{ Weak compactness}\label{subs:wk}

\begin{lem}\label{compactness1}
 Assume that (H1)--(H6) hold and fix $T>0$. Then we have: \\
(i) There is $L(T)>0$ (depending on $T$) such that
\begin{align*}
\int_{0}^{\infty}(1+x)f^n(x,t)dx\leq L(T)\ \ \text{for}\ \  n\ge 1\ \ \text{and all} \ \ t\in [0,T],\\
\end{align*}
(ii) For any $\varepsilon> 0$ there exists $R_\varepsilon>0$ such that for all $t\in[0,T]$\\
\begin{align*}
\sup_{n\ge 1} \left\{ \int_{R_\varepsilon}^{\infty}f^n(x,t)dx \right\}\leq \varepsilon,
\end{align*}
(iii) given $\varepsilon > 0$ there exists  $\delta_\varepsilon>0$ such that, for every measurable set $E$ of $]0,\infty[$ with $|E|\leq \delta_\varepsilon$, $n \ge 1$, and $t\in [0,T]$,\\
\begin{align*}
\int_{E}f^n(x,t)dx< \varepsilon.
\end{align*}
\end{lem}

\begin{proof}
(i) Let $n\ge 1$ and $t\in [0,T]$. Integrating (\ref{trunc1}) with respect to $x$ over $]0,1[$ and using Fubini's Theorem, we have
\begin{align*}
 \frac{d}{dt} \int_{0}^{1}f^n(x,t)dx = & -\frac{1}{2} \int_{0}^{1}\int_{0}^{1-x} K_n(x,y)f^n(x,t)f^n(y,t)dydx \\
&- \int_{0}^{1}\int_{1-x}^{n-x} K_n(x,y)f^n(x,t)f^n(y,t)dydx \\
&+\int_{0}^{1}\int_{x}^{n}b(x,y) S(y)f^n(y,t)dydx - \int_{0}^{1}S(x)f^n(x,t)dx.
\end{align*}
Since $K_n$, $f^n$, and $S$ are non-negative and $\Gamma$ satisfies (\ref{Selection rate defi1}), we have
\begin{align*}
\frac{d}{dt} \int_{0}^{1}f^n(x,t)dx \leq & \int_{0}^{1}\int_{x}^{n}b(x,y) S(y)f^n(y,t)dydx \\
                        = & \int_{0}^{1}\int_{x}^{1}\Gamma(y,x) f^n(y,t)dydx + \int_{0}^{1}\int_{1}^{n} \Gamma(y,x) f^n(y,t)dydx,
\end{align*}
Using Fubini's Theorem and (H5) (with $R=1$ and $E=]0,1[$) in the first term of the right-hand side and (H4) (with $R=1$) in the second one, we obtain
\begin{align}
 \frac{d}{dt} \int_{0}^{1}f^n(x,t)dx \leq&\int_{0}^{1}f^n(y,t)\int_{0}^{y}\Gamma(y,x) dxdy + k(1)\int_{0}^{1}\int_{1}^{n} y f^n(y,t)dydx \nonumber\\
\leq & \omega(1,1)\ \int_{0}^{1} f^n(x,t)dx+ k(1)\ \|f^n(t)\|_x .\label{big1}
\end{align}
Recalling that $\|f^n(t)\|_x=\|f^n(0)\|_x\leq \|f_0\|$ for $t\ge 0$ by (\ref{trunc mass1}), we readily deduce from (\ref{big1}) that
\begin{align*}
\frac{d}{dt}\int_{0}^{1}f^n(x,t)dx &\leq\omega(1,1) \int_{0}^{1}f^n(y,t)dy+ k(1)\|f_0\|.
\end{align*}
Integrating with respect to time, we end up with
\begin{align*}
\int_{0}^{1}f^n(x,t)dx \leq \|f_0\| \left( 1 + \frac{k(1)}{\omega(1,1)} \right) \exp(\omega(1,1)t), \ \ t\in [0,T].
\end{align*}
Using (\ref{trunc mass1}) again we may estimate
\begin{align*}
\int_{0}^{\infty}(1+x)f^n(x,t)dx &= \int_{0}^{1}f^n(x,t)dx+\int_{1}^{n}f^n(x,t)dx+\int_{0}^{n}xf^n(x,t)dx\\
&\leq \int_{0}^{1}f^n(x,t)dx+\int_{1}^{n}xf^n(x,t)dx+ \|f_0\|\\
&\leq \|f_0\|\ \left[ \left( 1 + \frac{k(1)}{\omega(1,1)} \right) \exp(\omega(1,1)T)+2 \right]=:L(T).
\end{align*}

\medskip

(ii) For $\varepsilon> 0$, set $R_\varepsilon:=\|f_0\|/ \varepsilon$. Then, by (\ref{trunc mass1}), for each $n\ge 1$ and for all $t\in[0,T]$ we have
$$
\int_{R_\varepsilon}^{\infty}f^n(x,t)dx \leq \frac{1}{R_\varepsilon}\int_{R_\varepsilon}^{\infty}xf^n(x,t)dx \leq \frac{\|f_0\|}{R_\varepsilon}< \varepsilon.
$$

\medskip

(iii) Fix $R>0$. For $n\ge 1$, $\delta\in (0,1)$, and $t\in [0,T]$, we define
\begin{align*}
p^n(\delta,t)=\sup \left\{\int_{0}^{R} \mathds{1}_{E}(x)f^n(x,t)dx\ :\ E\subset ]0,R[ \ \ \text{and} \ \
|E|\leq \delta \right\}.
\end{align*}

Consider a measurable subset $E\subset ]0,R[$ with $|E|\leq \delta$. For $n\ge 1$ and $t\in [0,T]$, it follows from the non-negativity of $f^n$, (\ref{Selection  rate defi1}) and (\ref{trunc1})-(\ref{trunc in1}) that

\begin{equation}
\frac{d}{dt} \int_{0}^{R} \mathds{1}_{E}(x) f^n(x,t)dx \leq \frac{1}{2}\ I_1^n(t) + I_2^n(t) + I_3^n(t), \label{evian}
\end{equation}
where
\begin{align*}
I_1^n(t) := & \int_{0}^{R}\mathds{1}_{E}(x)\int_{0}^{x} K_n(x-y,y)f^n(x-y,t)f^n(y,t)dydx , \\
I_2^n(t) := & \int_{0}^{R}\mathds{1}_{E}(x)\int_{x}^{R} \Gamma(y,x) f^n(y,t)dydx , \\
I_3^n(t) := & \int_{0}^{R}\mathds{1}_{E}(x)\int_{R}^{\infty} \Gamma(y,x) f^n(y,t)dydx .
\end{align*}
First, applying Fubini's Theorem to $I_1^n(t)$ gives
\begin{align*}
I_1^n(t) = & \int_{0}^{R} f^n(y,t) \int_{y}^{R}\mathds{1}_{E}(x)K_n(y,x-y)f^n(x-y,t)dxdy \\
= & \int_{0}^{R}f^n(y,t)\int_{0}^{R-y}\mathds{1}_{E}(x+y)K_n(y,x)f^n(x,t)dxdy.
\end{align*}
Setting $-y+ E := \{z>0\ :\ z=-y+x \ \ \text{for some}\ \ x\in E\}$, it follows from (H2) and the above identity that
$$
I_1^n(t) \le k_1^2 (1+R)^{\mu} \int_{0}^{R}(1+y)^{\mu}f^n(y,t) \int_{0}^{R} f^n(x,t)\mathds{1}_{{-y+E}\cap ]0,R-y[}(x)dxdy.
$$
Since ${-y+E}\cap ]0,R-y[\subset ]0,R[$ and $|{-y+E}\cap ]0,R-y[|\le|-y+E|=|E|\le\delta$, we infer from the definition of
$p^n(\delta,t)$ and Lemma~\ref{compactness1}~(i) that
$$
I_1^n(t)\le k_1^2 (1+R)^{\mu} \left( \int_{0}^{R}(1+y)^{\mu}f^n(y,t) dy \right)\ p^n(\delta,t) \le k_1^2 L(T) (1+R)^{\mu} p^n(\delta,t).$$
Next, applying Fubini's Theorem to $I_2^n(t)$ and using (H5) and Lemma~\ref{compactness1}~(i) give
\begin{align*}
I_2^n(t) =& \int_{0}^{R}f^n(y,t)\int_{0}^{y}\mathds{1}_{E}(x)\Gamma(y,x)dxdy \leq \omega(R,|E|)\ \int_{0}^{R} f^n(y,t)dy \leq L(T) \omega(R,|E|).
\end{align*}
Finally, owing to (H4) and (\ref{trunc mass1}), we have
\begin{align*}
I_3^n(t) \leq & k(R)\ \int_{0}^{R} \int_{R}^{\infty} \mathds{1}_{E}(x) y^\theta f^n(y,t) dydx \leq k(R) R^{\theta-1}\ |E|\ \int_{R}^{\infty} y f^n(y,t)dy \\
\leq & k(R) R^{\theta-1}\ \|f_0\|\ |E| \leq k(R) R^{\theta-1}\ \|f_0\|\ \delta.
\end{align*}
Collecting the estimates on $I_j^n(t)$, $1\le j \le 3$, we infer from (\ref{evian}) that there is $C_1(R,T)>0$ such that
$$
\frac{d}{dt} \int_{0}^{R} \mathds{1}_{E}(x)f^n(x,t)dx \leq C_1(R,T)\ \left( p^n(\delta,t) + \omega(R,\delta)+ \delta \right).
$$
Integrating with respect to time and taking the supremum over all $E$ such that $E\subset ]0,R[$ with $|E|\leq \delta$ give
$$
p^n(\delta,t) \leq p^n(\delta,0) + T C_1(R,T) [\omega(R,\delta)+\delta]+C_1(R,T)\int_{0}^{t}p^n(\delta,s)ds, \qquad t\in [0,T].
$$
By Gronwall's inequality (see e.g. \cite[p. 310]{Walter:1998}), we obtain
\begin{align}\label{main221}
p^n(\delta,t)\leq \left[ p^n(\delta,0)+ T C_1(R,T) (\omega(R,\delta)+\delta) \right] \exp{\{C_1(R,T)t\}}, \qquad t\in [0,T].
\end{align}
Now, since $f^n(x,0)\le f_0(x)$ for $x>0$, the absolute continuity of the integral guarantees that $\sup_{n}\{p^n(\delta,0)\} \to 0$ as $\delta \to 0$ which implies, together with (H5) and (\ref{main221}) that
$$
\lim_{\delta\to 0} \sup_{n\ge 1, t\in [0,T]}{\{p^n(\delta,t)\}} =0.
$$
Lemma~\ref{compactness1}~(iii) is then a straightforward consequence of this property and Lemma~\ref{compactness1}~(i).
\end{proof}

\medskip

Lemma~\ref{compactness1} and the \textit{Dunford-Pettis Theorem} imply that, for each $t\in[0,T]$, the sequence of functions $(f^n(t))_{n\ge 1}$ lies in a weakly relatively compact set of $L^1]0,\infty[$ which does not depend on $t\in [0,T]$.

\subsection{Equicontinuity in time}\label{subs:equit}

Now we proceed to show the time equicontinuity of the sequence $(f^n)_{n\in \mathds{N}}$. Though the coagulation terms can be handled as in \cite{GIRI:2010EXT,Laurencot:2000,Stewart:1990I}, we sketch the proof below for the sake of completeness. Let $T>0$, $\varepsilon > 0$, and $\phi \in L^{\infty}]0,\infty[$ and consider $s,t\in [0,T]$ with $t\geq s$. Fix $R>1$ such that
\begin{align}\label{m1}
\frac{2L(T)}{R}< \frac{\varepsilon}{2},
\end{align}
the constant $L(T)$ being defined in Lemma~\ref{compactness1}~(i). For each $n$, by Lemma~\ref{compactness1}~(i),
\begin{align}\label{mm1}
\int_{R}^{\infty}|f^n(x,t)-f^n(x,s)|dx\leq \frac{1}{R}\int_{R}^{\infty}x\{f^n(x,t)+f^n(x,s)\}dx\leq \frac{2L(T)}{R}.
\end{align}
By (\ref{trunc1}), (\ref{m1}), and (\ref{mm1}), we get
\begin{align}\label{equicontinuity1}
& \left| \int_{0}^{\infty}\phi(x)\{f^n(x,t)-f^n(x,s)\}dx \right| \nonumber\\
\leq & \left| \int_{0}^{R}\phi(x)\{f^n(x,t)-f^n(x,s)\}dx \right| + \int_{R}^{\infty}|\phi(x)||f^n(x,t)-f^n(x,s)|dx \nonumber\\
\leq & \|\phi\|_{L^{\infty}} \int_{s}^{t} \left[ \frac{1}{2}\int_{0}^{R}\int_{0}^{x} K_n(x-y,y)f^n(x-y,\tau)f^n(y,\tau)dydx \right. \nonumber\\
&+\int_{0}^{R} \int_{0}^{n-x} K_n(x,y)f^n(x,\tau)f^n(y,\tau)dydx +\int_{0}^{R}\int_{x}^{n}b(x,y) S(y)f^n(y,\tau)dydx\nonumber\\
& \left. +\int_{0}^{R}S(x)f^n(x,\tau)dx \right] d{\tau} + \|\phi\|_{L^{\infty}}\frac{\varepsilon}{2}.
\end{align}

By Fubini's Theorem, (H2), and Lemma~\ref{compactness1}~(i), the first term of the right-hand side of (\ref{equicontinuity1}) may be estimated as follows:
\begin{align*}
\frac{1}{2}\int_{0}^{R}\int_{0}^{x}K_n(x-y,y)&f^n(x-y,\tau)f^n(y,\tau)dy dx\\
&=\frac{1}{2}\int_{0}^{R}\int_{y}^{R} K_n(x-y,y)f^n(x-y,\tau)f^n(y,\tau)dx dy \\
&=\frac{1}{2}\int_{0}^{R}\int_{0}^{R-y}K_n(x,y)f^n(x,\tau)f^n(y,\tau)dx dy \\
&\leq \frac{k_{1}^2}{2}\int_{0}^{R}\int_{0}^{R-y}(1+x)^{\mu}(1+y)^{\mu}f^n(x,\tau)f^n(y,\tau)dy dx\\
&\leq \frac{k_{1}^2 L(T)^2}{2}.
\end{align*}
Similarly, for the second term of the right-hand side of (\ref{equicontinuity1}), it follows from (H2) that
\begin{align*}
\int_{0}^{R}\int_{0}^{n-x}K_n(x,y)f^n(x,\tau)f^n(y,\tau)dy dx &\leq k_{1}^2\int_{0}^{R}\int_{0}^{n-x}(1+x)^{\mu}(1+y)^{\mu}f^n(x,\tau)f^n(y,\tau)dy dx \\
&\leq k_{1}^2 L(T)^2.
\end{align*}
For the third term of the right-hand side of (\ref{equicontinuity1}), we use Fubini's Theorem, (H4), (H5), and Lemma~\ref{compactness1}~(i) to obtain
\begin{align*}
\int_{0}^{R}\int_{x}^{n}b(x,y) &S(y)f^n(y,\tau)dydx \\
\le & \int_{0}^{R}\int_{0}^{y}\Gamma(y,x)f^n(y,\tau)dxdy +\int_{0}^{R}\int_{R}^{\infty}\Gamma(y,x)f^n(y,\tau)dydx\\
\leq & \int_{0}^{R}f^n(y,\tau)\int_{0}^{y}\mathds{1}_{]0,R[}(x)\Gamma(y,x)dxdy + k(R)\ \int_{0}^{R}\int_{R}^{\infty} y^\theta f^n(y,\tau)dydx\\
\leq &\omega(R,R) \int_{0}^{R}f^n(y,\tau)dy+ k(R) \int_{0}^{R}\int_{R}^{\infty} yf^n(y,\tau)dydx\\
\leq & [\omega(R,R) + R k(R)]\  L(T).
\end{align*}
Finally, the fourth term of the right-hand side of (\ref{equicontinuity1}) is estimated with the help of (H6) and Lemma~\ref{compactness1}~(i)  and we get
\begin{align*}
\int_{0}^{R}S(x)f^n(x,t)dx \leq \|S\|_{L^{\infty}]0,R[} L(T).
\end{align*}
Collecting the above estimates and setting
$$
C_2(R,T) = \frac{3k_{1}^2 L(T)^2}{2}+\left\{ \omega(R,R) +  R k(R)+\|S\|_{L^{\infty}]0,R[} \right\}\ L(T)
$$
the inequality (\ref{equicontinuity1}) reduces to
\begin{equation}\label{finalequi1}
\left| \int_{0}^{\infty}\phi(x)\{f^n(x,t)-f^n(x,s)\} dx \right| \leq C_2(R,T)\ \|\phi\|_{L^{\infty}}\ (t-s) + \|\phi\|_{L^{\infty}}\frac{\varepsilon}{2}<\|\phi\|_{L^{\infty}}\varepsilon,
\end{equation}
whenever $t-s<\delta$ for some suitably small $\delta>0$. The estimate (\ref{finalequi1}) implies the time equicontinuity of the family $\{f^n(t), t\in[0,T]\}$ in $L^1]0,\infty[$. Thus, according to a refined version of the \textit{Arzel\`{a}-Ascoli Theorem}, see \cite[Theorem 2.1]{Stewart:1990I}, we conclude that there exist a subsequence (${f^{n_k}}$) and a non-negative function $f\in L^\infty(]0,T[;L^1]0,\infty[)$ such that
\begin{equation}
\lim_{n_k\to\infty} \sup_{t\in [0,T]}{\left\{ \left| \int_0^\infty  \left\{ f^{n_k}(x,t) - f(x,t) \right\}\ \phi(x)\ dx \right| \right\}} = 0, \label{vittel}
\end{equation}
for all $T>0$ and $\phi \in L^\infty]0,\infty[$. In particular, it follows from the non-negativity of $f^n$ and $f$, (\ref{trunc mass1}), and (\ref{vittel}) that, for $t\ge 0$ and $R>0$,
$$
\int_{0}^{R}x f(x,t)dx=\lim_{n_k\to \infty}\int_{0}^{R}x f^{n_k}(x,t)dx\leq \|f_0\|_x< \infty.
$$
Letting $R\to \infty$ implies that $\|f(t)\|_x\le \|f_0\|_x$ and thus $f(t)\in X^+$.

\subsection{Passing to the limit}\label{subs:limit}

Now we have to show that the limit function $f$ obtained in (\ref{vittel}) is actually a weak solution to (\ref{cfe1})-(\ref{in1}).
To this end, we shall use weak continuity and convergence properties of some operators which define now: for $g\in X^+$, $n\ge 1$, and $x\in ]0,\infty[$, we put
\begin{align*}
Q_{1}^n(g)(x)=\frac{1}{2}\int_{0}^{x}K_n(x-y,y)g(x-y)g(y)dy,\ \ & \ \ Q_{2}^n(g)(x)=\int_{0}^{n-x}K_n(x,y)g(x)g(y)dy,\\
Q_{1}(g)(x)=\frac{1}{2}\int_{0}^{x}K(x-y,y)g(x-y)g(y)dy, \ \ &\ \  Q_{2}(g)(x)=\int_{0}^{\infty}K(x,y)g(x)g(y)dy,\\
Q_{3}(g)(x)=S(x)g(x),\ \ & \ \ Q_{4}(g)(x)=\int_{x}^{\infty}b(x,y)S(y)g(y)dy,
\end{align*}
and $Q^n=Q_1^n-Q_2^n-Q_3+Q_4$, $Q=Q_1-Q_2-Q_3+Q_4$.

We then have the following result:

\begin{lem}\label{convergence lemma1}
Let $(g^{n})_{n\in \mathds{N}}$ be a bounded sequence in $X^+$, $||g^n||\leq L$, and  $g\in X^+$ such that $g^n\rightharpoonup g$ in $L^1]0,\infty[$ as $n\to \infty $. Then, for each $R> 0$ and $i\in\{1,\ldots,4\}$, we have
\begin{equation}
Q_i^n(g^n)\rightharpoonup Q_i(g)\ \ \text{in} \ \ L^1]0,R[\ \ \text{as}\ \ n\to \infty. \label{luchon}
\end{equation}
\end{lem}

\begin{proof}
The proof of (\ref{luchon}) for $i=1,2$ is the same as that in \cite{GIRI:2010EXT,Stewart:1990I} to which we refer. The case $i=3$ is obvious since $\phi S$ belongs to $L^\infty]0,R[$ by (H6) and (\ref{luchon}) follows at once from the weak convergence of $(g^n)$ in $L^1]0,\infty[$. For $i=4$, we consider $\phi \in L^{\infty}]0,R[$ and use (\ref{Selection rate defi1}) and Fubini's Theorem to compute, for $r>R$,
\begin{align*}
\left| \int_{0}^{R}\phi(x) \{Q_4(g^n)(x)-Q_4(g)(x)\} dx\right| = & \left| \int_{0}^{R}\int_{x}^{\infty}\phi(x)S(y)b(x,y)\{g^n(y)-g(y)\}dydx \right| \\
\leq &\left| \int_{0}^{R}\int_{0}^{y}\phi(x)S(y)b(x,y) \{g^n(y)-g(y)\}dxdy \right| \\
& + \left| \int_{R}^{\infty}\int_{0}^{R}\phi(x)\Gamma(y,x)\{g^n(y)-g(y)\}dxdy \right|.
\end{align*}
This can be further written as
\begin{equation}\label{defj}
\left| \int_{0}^{R}\phi(x) \{Q_4(g^n)(x)-Q_4(g)(x)\} dx\right| = J_1^n + J_2^n(r) + J_3^n(r),
\end{equation}
with
\begin{align*}
J_1^n = & \left| \int_{0}^{R}\{g^n(y)-g(y)\}\int_{0}^{y}\phi(x)S(y)b(x,y) dxdy \right| \\
J_2^n(r) = & \left| \int_{R}^{r}\{g^n(y)-g(y)\}\int_{0}^{R}\phi(x)\Gamma(y,x)dxdy \right| \\
J_3^n(r) = & \left| \int_{r}^{\infty}\{g^n(y)-g(y)\}\int_{0}^{R}\phi(x)\Gamma(y,x)dxdy \right|.
\end{align*}

We use (H6) and (\ref{N1}) to observe that, for $y \in ]0,R[$,
\begin{align*}
\left| \int_{0}^{y}\phi(x) S(y) b(x,y) dx \right| &\leq \|S\|_{L^{\infty}]0,R[} \|\phi\|_{L^{\infty}]0,R[} \int_{0}^{y}b(x,y)dx\\
& \leq N \|S\|_{L^{\infty}]0,R[} \|\phi\|_{L^{\infty}]0,R[}.
\end{align*}
This shows that the function $y\mapsto \int_{0}^{y}\phi(x)\Gamma(y,x)dx$ belongs to $L^{\infty}]0,R[$. Since $g^n\rightharpoonup g$ in $L^1]0,\infty[$ as $n \to \infty$, it thus follows that
\begin{align}\label{j1}
\lim_{n \to \infty} J^n_1=0.
\end{align}
We next infer from (H4) that, for $y \in ]0,R[$,
\begin{equation}
\left| \int_0^R \phi(x) \Gamma(y,x) dx \right| \le k(R)\ y^\theta\ \int_0^R \phi(x)\ dx \le R k(R) \|\phi\|_{L^\infty]0,R[}\ y^\theta. \label{alet}
\end{equation}
On the one hand, (\ref{alet}) guarantees that the function $y\mapsto \int_{0}^{R}\phi(x)\Gamma(y,x)dx$ belongs to $L^{\infty}]R,r[$ and the weak convergence of $(g^n)$ to $g$ in $L^1]0,\infty[$ entails that
\begin{align}\label{j2}
\lim_{n \to \infty} J^n_2(r)=0  \ \ \text{ for all }\ \ r>R.
\end{align}
On the other hand, we deduce from (\ref{alet}) and the boundedness of $(g^n)$ and $g$ in $X^+$ that
\begin{align*}
\left| \int_{r}^{\infty}\{g^n(y)-g(y)\}\int_{0}^{R} \phi(x)\Gamma(y,x)dxdy \right| \leq & Rk(R) \|\phi\|_{L^{\infty}]0,R[}\int_{r}^{\infty} y^{\theta}\{g^n(y)+g(y)\}dy\\
\leq & \frac{R k(R) (L+\|g\|)}{r^{1-\theta}} \|\phi\|_{L^{\infty}]0,R[}
\end{align*}
which is asymptotically small (as $r \to \infty$) uniformly with respect to $n$. We thus conclude that
\begin{align}\label{j3}
\lim_{r \to \infty} \sup_{n\ge 1}\{J^n_3(r)\}=0.
\end{align}
Substituting (\ref{j1}) and (\ref{j2}) into (\ref{defj}), we obtain
$$
\limsup_{n\to\infty} \left| \int_{0}^{R}\phi(x)\{Q_4(g^n)(x)-Q_4(g)(x)\}dx \right| \le \sup_{n\ge 1}\{J^n_3(r)\}
$$
for all $r>R$. Owing to (\ref{j3}), we may let $r\to\infty$ and conclude that (\ref{luchon}) holds true for $i=4$ thanks to the arbitrariness of $\phi$ and the proof of Lemma~\ref{convergence lemma1} is complete.
\end{proof}

\subsection{Existence}\label{subs:exist}
Now we are in a position to prove the main result.

\begin{proof}[Proof of Theorem~\ref{existmain theorem1}]
Fix $R>0$, $T>0$, and consider $t\in [0,T]$ and $\phi\in L^{\infty}]0,R[$. Owing to Lemma~\ref{convergence lemma1}, we have for each $s\in[0,t]$,
\begin{align}\label{main convergence1}
\int_{0}^{R}\phi(x)\{Q^{n_k}(f^{n_k}(s))(x)-Q(f(s))(x)\}dx \to 0 \ \ as \ \ n_k\to \infty.
\end{align}
Arguing as in Section~\ref{subs:equit}, it follows from (H2), (H4)--(H6), and Lemma~\ref{compactness1}~(i) that there is $C_3(R,T)>0$ such that, for $n\ge 1$, and $s\in [0,t]$, we have
\begin{equation}\label{bound on the diff1}
\left| \int_{0}^{R} \phi(x) Q^n(f^n(s))(x) dx\right| \le C_3(R,T)\ \|\phi\|_{L^{\infty}]0,R[}.
\end{equation}
Since the right-hand side of (\ref{bound on the diff1}) is in $L^1]0,t[$, it follows from (\ref{main convergence1}), (\ref{bound on the diff1}) and the dominated convergence theorem that
\begin{align}\label{main convergence with t1}
\left| \int_{0}^{t}\int_{0}^{R}\phi(x)\{Q^{n_k}(f^{n_k}(s))(x)-Q(f(s))(x)\}dx ds \right| \to 0 \ \ \text{as} \ \ n_k\to \infty.
\end{align}
Since $\phi$ is arbitrary in $L^\infty]0,R[$,  Fubini's Theorem and (\ref{main convergence with t1}) give
\begin{align}\label{conv of Q with t1}
\int_{0}^{t}Q^{n_k}(f^{n_k}(s))ds\rightharpoonup \int_{0}^{t}Q(f(s))ds\ \ \text{in} \ \ L^1]0,R[\ \ \text{as}\ \ n_k\to \infty.
\end{align}
It is then straightforward to pass to the limit as $n_k\to\infty$ in (\ref{trunc1})-(\ref{trunc in1}) and conclude that  $f$ is a solution to (\ref{cfe1})-(\ref{in1}) on $[0,\infty[$ (since $T$ is arbitrary). This completes the proof of Theorem~\ref{existmain theorem1}.
\end{proof}

\begin{rmkk}\label{re:N}
It is worth pointing out that the assumption (\ref{N1}) $\int_0^y b(x,y) dx = N$ is only used to prove (\ref{j1}) and it is clear from that proof that the assumption
$$
\sup_{y\in ]0,R[} \int_0^y b(x,y) dx < \infty \ \ \text{ for all } \ \ R>0
$$
is sufficient. Thus, Theorem~\ref{existmain theorem1} is actually valid under this weaker assumption.
\end{rmkk}

\section*{Acknowledgments}

A.K. Giri would like to thank International Max-Planck Research School, Magdeburg, Germany and FWF Austrian Science Fund for their support. Part of this work was done while Ph. Lauren\c{c}ot enjoys the hospitality and support of the Isaac Newton Institute for Mathematical Sciences, Cambridge, UK. 

\bibliographystyle{plain}

\begin{thebibliography}{1}

\bibitem{Ash:1972} R.B. Ash. \textit{Measure, Integration and Functional Analysis}.
Academic Press, New York, 1972.

\bibitem{Edwards:1965} R.E. Edwards. \textit{Functional Analysis: Theory and Applications}.
Holt, Rinehart and Winston, New York, 1965.

\bibitem{GIRI:2010EXT} A.K. Giri, J. Kumar, and G. Warnecke. The continuous coagulation equation with multiple fragmentation.
\textit{J. Math. Anal. Appl.}, 374: 71--87, 2011.

\bibitem{Lamb:2004} W. Lamb. Existence and uniqueness results for the continuous coagulation and fragmentation equation. \textit{Math. Methods Appl. Sci.}, 27: 703--721, 2004.

\bibitem{Laurencot:2000} Ph. Lauren{\c c}ot. On a class of continuous coagulation-fragmentation equations. \textit{J. Differential Equations}, 167: 245--274, 2000.

\bibitem{PhL:2002} Ph. Lauren{\c c}ot. The discrete coagulation equation with multiple fragmentation. \textit{Proc. Edinburgh Math. Soc.}, 45: 67--82, 2002.

\bibitem{LM:2004} Ph. Lauren{\c c}ot and S. Mischler. On coalescence equations and related models. In P. Degond, L. Pareschi, and G. Russo, editors, \textit{Modeling and Computational Methods for Kinetic Equations}, pages 321--356. Birkh\"auser, Boston, 2004.

\bibitem{McGradyZiff:1987} E.D. McGrady and R.M. Ziff. ``Shattering'' transition in fragmentation. \textit{Phys. Rev. Lett.}, 58: 892--895, 1987.

\bibitem{McLaughlin:1997II} D.J. McLaughlin, W. Lamb, and A.C. McBride. An existence and uniqueness result for a coagulation and multiple-fragmentation equation. \textit{SIAM J. Math. Anal.}, 28: 1173--1190, 1997.

\bibitem{Melzak:1957} Z.A. Melzak. A scalar transport equation. \textit{Trans. Amer. Math. Soc.}, 85: 547--560, 1957.

\bibitem{Peterson:1986} T.W. Peterson. Similarity solutions for the population balance equation describing particle fragmentation. \textit{Aerosol. Sci. Technol.}, 5: 93--101, 1986.

\bibitem{Stewart:1990I} I.W. Stewart. A global existence theorem for the general coagulation-fragmentation equation with unbounded kernels. \textit{Math. Methods Appl. Sci.}, 11: 627--648, 1989.

\bibitem{Walker:2002} C. Walker. Coalescence and breakage processes. \textit{Math. Methods Appl. Sci.}, 25: 729--748, 2002.

\bibitem{Walter:1998} W. Walter. \textit{Ordinary Differential Equations}. Springer-Verlag New York, USA, 1st edition, 1998.

\end{thebibliography}


\end{document}